\documentclass[11pt,twoside]{amsart}
\usepackage{amsmath,amssymb,amsthm,stmaryrd}
\usepackage{mathrsfs}
\usepackage[dvipsnames]{xcolor}
\usepackage{hyperref}

\makeatletter
\@namedef{subjclassname@2020}{%
Mathematics Subject Classification}
\makeatother

\numberwithin{equation}{section}

\newtheorem{Thm}[equation]{Theorem}
\newtheorem{Lem}[equation]{Lemma}
\newtheorem{Pro}[equation]{Proposition}
\newtheorem{Cor}[equation]{Corollary}
\theoremstyle{definition}
\newtheorem{Rem}[equation]{Remark}

\newcommand{\R}{{\mathbb R}}
\newcommand{\Z}{{\mathbb Z}}

\newcommand{\bI}{{\mathbf I}}
\newcommand{\M}{{\mathbf M}}
  
\newcommand{\cB}{{\mathscr B}}
\newcommand{\cD}{{\mathscr D}}

\newcommand{\cP}{{\mathscr P}}
\newcommand{\cZ}{{\mathscr Z}}

\newcommand{\del}{\delta}
\newcommand{\Del}{\Delta}
\newcommand{\eps}{\epsilon}
\newcommand{\gam}{\gamma}
\newcommand{\Gam}{\Gamma}

\newcommand{\lam}{\lambda}
\newcommand{\Lam}{\Lambda}
\newcommand{\om}{\omega}
\newcommand{\Om}{\Omega}
\renewcommand{\rho}{\varrho}
\newcommand{\sig}{\sigma}
\newcommand{\Sig}{\Sigma}

\newcommand{\area}{\operatorname{area}}
\newcommand{\asdim}{\operatorname{as\,dim}}
\newcommand{\asNdim}{asymptotic Nagata dimension}
\newcommand{\bb}[1]{\llbracket #1\rrbracket} 
\newcommand{\on}{\mathop{\mbox{\rule{0.1ex}{1.2ex}\rule{1.1ex}{0.1ex}}}}
\newcommand{\const}{\operatorname{const}}
\renewcommand{\d}{\partial}
\newcommand{\diam}{\operatorname{diam}}
\newcommand{\id}{\operatorname{id}}
\newcommand{\cs}{\text{\rm c}} 
\newcommand{\spt}{\operatorname{spt}}
\newcommand{\CAT}{\operatorname{CAT}}
\newcommand{\Fill}{\operatorname{Fill}}

\newcommand{\mesh}{\operatorname{mesh}}
\newcommand{\sub}{\subset}
\newcommand{\sm}{\setminus}
\newcommand{\ol}{\overline}

\newcommand{\es}{\emptyset}
\newcommand{\Lip}{\operatorname{Lip}}
\newcommand{\vol}{\operatorname{vol}}

\begin{document}

\title[Isoperimetric inequalities in Hadamard spaces]{Isoperimetric inequalities in Hadamard spaces of asymptotic rank two}
\author{Urs Lang}
\address{Department of Mathematics\\ETH Z\"urich\\R\"amistrasse 101
  \\8092 Z\"urich\\Switzerland}
\email{lang@math.ethz.ch}
\author{Stephan Stadler}
\address{Max Planck Institute for Mathematics\\Vivatsgasse 7\\53111 Bonn\\Germany}
\email{stadler@mpim-bonn.mpg.de}
\author{David Urech}
\address{Department of Mathematics\\ETH Z\"urich\\R\"amistrasse 101
  \\8092 Z\"urich\\Switzerland}
\email{david.urech@math.ethz.ch}

\date{July 2, 2026}

\begin{abstract}
Gromov's isoperimetric gap conjecture for Hadamard spaces states that cycles in dimensions greater than or equal to the asymptotic rank admit linear isoperimetric filling inequalities, as opposed to the inequalities of Euclidean type in lower dimensions.
In the case of asymptotic rank $2$, recent progress was made by Dru\c{t}u--Lang--Papasoglu--Stadler who established a homotopical inequality for Lipschitz $2$-spheres with exponents arbitrarily close to $1$. 
We prove a homological inequality of the same type for general cycles in dimensions at least $2$, assuming that the ambient space has finite linearly controlled asymptotic dimension. 
This holds in particular for all Hadamard $3$-manifolds and finite-dimensional $\CAT(0)$ cube complexes. 
\end{abstract}

\maketitle


\section{Introduction}

The classical isoperimetric inequality in $\R^{k+1}$ bounds the volume of a domain in terms of the $k$-dimensional area of its boundary or, specifically, by the volume of a ball with equal perimeter. 
For $k$-dimensional boundaries in higher codimension, isoperimetric filling inequalities assert the existence of a $(k+1)$-dimensional spanning surface satisfying a similar relation. Inequalities of this type arise in a much wider context and are of fundamental interest in metric geometry and geometric group theory, where their growth rate serves as an
asymptotic invariant of the ambient space.

In this paper, we focus on spaces of non-positive curvature, Hadamard manifolds, or their metric companions, $\CAT(0)$ spaces, complete geodesic metric spaces satisfying Euclidean triangle comparisons (see~\cite{Bal, BriH}). A basic result due to Gromov~\cite[3.4.C, Remarks~(a) and~(c)]{Gro-FRR} shows that an {\em isoperimetric inequality of Euclidean type}
holds for integral $k$-cycles in any Hadamard $n$-manifold, independently of $n$: every $k$-cycle $T$ admits a filling by a $(k+1)$-chain $V$ with boundary $\d V = T$ and volume
\[
\vol_{k+1}(V) \le \const \cdot \vol_k(T)^{1+1/k}
\]
for some constant depending only on~$k$. 
(This was later shown to hold with the optimal constant $\vol_{k+1}(B^{k+1})/\vol_k(\partial B^{k+1})^{1+1/k}$, where $B^{k+1}$ is the Euclidean unit ball, in a few cases, notably if $X = \R^n$~\cite{Alm-OI} or if $k=2$~\cite{Sch}.)
Wenger~\cite[Theorem~1.2]{Wen-EII} extended the above (non-sharp) inequality to a wide class of metric spaces including all $\CAT(0)$ spaces.
This result is restated in Theorem~\ref{Thm:eucl-isop} below.

We are interested in improvements of this inequality deriving from more specific geometric properties of the ambient Hadamard space. 
If the curvature is bounded above by a negative constant, then a geodesic cone construction yields a {\em linear isoperimetric inequality}
\[
\vol_{k+1}(V) \le \const \cdot \vol_k(T);
\]
see~\cite[Theorem~1.7]{Wen-FI}.
In general, one expects a phase transition similar to what occurs in Riemannian symmetric spaces $X$ of rank $\nu$: in dimensions $k < \nu$, cycles in~$X$ satisfy a Euclidean-type inequality, while for $k \ge \nu$, a linear inequality holds (compare~\cite[p.~105, (b$'_1$)]{Gro-AI}, \cite[Theorem~1]{Leu}, and \cite[Theorem~1, Remarks~9 and~10]{Isl}). 
In~\cite[p.~128, (b)]{Gro-AI}, Gromov stated this conjectural isoperimetric gap for $\CAT(0)$ spaces $X$ in the form of an equality between the critical dimension $\nu$ and various notions of rank.
Subsequent work by Kleiner~\cite[Theorems~C and~7.1]{Kle} and an improved Euclidean (`sub-Euclidean') inequality due to Wenger~\cite[Theorem~1.2, Corollary~1.5]{Wen-AR} support the conjecture for a unified notion of {\em asymptotic rank}.
This quasi-isometry invariant equals the supremum of the dimensions of Euclidean subspaces isometrically embedded in any asymptotic cones of~$X$, or in $X$ itself in case $X$ is proper with cocompact isometry group. (See also~\cite{KleL, GolL} for further discussion).

The asymptotic rank $\nu$ is at most $1$ if and only if $X$ is Gromov hyperbolic~\cite[Corollary~1.3]{Wen-AR}. 
For $1$-cycles, every $\CAT(0)$ space $X$ satisfies a Euclidean (quadratic) isoperimetric inequality with the sharp constant $1/(4\pi)$, and an improved inequality with a strictly smaller constant implies that $X$ is Gromov hyperbolic and satisfies a linear isoperimetric inequality (see~\cite[p.~106]{Gro-HG} and~\cite{Wen-IC}). 
For cycles of dimension $k > \nu = 1$, \cite[Theorems~1 and~3]{Lan-LII} and \cite[p.~306]{LanP} provide linear fillings under some additional assumptions, while if $X$ is a general Gromov hyperbolic $\CAT(0)$ space, \cite[Theorem~1.6]{Wen-AR} yields for every $\del > 0$ an inequality of the form 
\[
\vol_{k+1}(V) \le \const_\del \cdot \vol_k(T)^{1+\del},
\]
in the following referred to as a {\em $\del$-isoperimetric inequality}.

For the case of higher asymptotic rank $\nu \ge 2$, recent progress on the isoperimetric gap conjecture was made in~\cite{DLPS}, where a homotopical $\delta$-isoperimetric inequality was established for proper $\CAT(0)$ spaces $X$ with asymptotic rank~$2$: 
every Lipschitz $2$-sphere $S$ in $X$ bounds a Lipschitz $3$-ball $B$ with
\[
\vol(B)\leq C\cdot\area(S)^{1+\delta},
\]
where the constant $C$ depends only on $X$ and $\delta > 0$.
More generally, it was shown that the $\del$-isoperimetric inequality remains valid for fillings of Lipschitz surfaces of higher genus by Lipschitz handlebodies, with a constant depending in addition on the genus. 

Here we now address the case of general $k$-cycles for $k \ge 2$ in a (not necessarily proper) $\CAT(0)$ space $X$ of asymptotic rank~$2$.  
We use the chain complex $\bI_{*,\cs}(X)$ of metric integral currents with compact support, which comprises all Lipschitz singular chains and has good compactness properties (see Section~\ref{sect:currents}).
To cope with the missing topological control of the cycles under consideration, we assume that $X$ has finite {\em asymptotic Nagata dimension}, a variant of Gromov's asymptotic dimension~\cite[1.E]{Gro-AI}, also known as {\em linearly controlled asymptotic dimension} (see Section~\ref{sect:nagata}). This will allow us to 
decompose cycles suitably.

\begin{Thm} \label{Thm:main}
Let $X$ be a $\CAT(0)$ space of asymptotic rank at most~$2$ and of finite asymptotic Nagata dimension. 
Then for every cycle $T \in \bI_{k,\cs}(X)$ in~$X$ of dimension
$k \ge 2$ and every $\del > 0$ there exists a $V \in \bI_{k+1,\cs}(X)$ with boundary $\d V = T$ and mass
\[
  \M(V) \le C\cdot\M(T)^{1 + \del}
\]
for some constant $C$ depending only on $X$, $k$, and $\del$.
\end{Thm}

Note that both conditions imposed on the underlying $\CAT(0)$
space remain invariant under quasi-isometries.
With the slightly stronger assumption of finite {\em Nagata dimension}, the result holds analogously for fillings of Lipschitz cycles by Lipschitz chains; see Remark~\ref{Rem:chains}. 

It would suffice to prove Theorem~\ref{Thm:main} in the case $k = 2$, by the aforementioned~\cite[Theorem~1.6]{Wen-AR}, which shows that the $\del$-isoperimetric inequality is passed on to higher-dimensional cycles.
However, we will not rely on the latter result, as our argument applies equally well to all $k \ge 2$.

We now list some instances covered by Theorem~\ref{Thm:main} (see 
Items (I)--(IV) in Section~\ref{sect:nagata} for more details).
\begin{enumerate}
\item[(A)] 
Three-dimensional Hadamard manifolds have Nagata dimension~$3$ by~\cite{JorL}.
In this case, for $k = 2$, the filling $V$ is uniquely determined by $T$
(as a distinct $V'$ would give rise to a non-zero $4$-current with boundary $V - V'$), 
and Theorem~\ref{Thm:main} amounts to an isoperimetric inequality for bounded domains with finite perimeter in $X$.
It is not known whether all Hadamard manifolds have finite (asymptotic) Nagata dimension. 
\item[(B)]
Finite-dimensional $\CAT(0)$ cube complexes have finite asymptotic Nagata dimension by the argument for the asymptotic dimension given in~\cite{Arz+}. 
In this case, or whenever $X$ is suitably triangulated, Theorem~\ref{Thm:main} implies an analogous statement for cellular chains via the deformation theorem of geometric measure theory; 
see~\cite[Theorem~A.2]{BasWY} (compare also the first part of the proof of 
Theorem~\ref{Thm:approx} below). 
\item[(C)] 
A Gromov hyperbolic $\CAT(0)$ space $X$ has finite asymptotic Nagata dimension provided that for some scale $r > 0$, every $2r$-ball in $X$ can be covered by a fixed finite number of $r$-balls~\cite{LanS-ND}. 
In particular, the product $X \times X'$ of two such spaces has asymptotic rank at most~$2$ and finite asymptotic Nagata dimension.   
\end{enumerate}

In view of~(A), for Hadamard $3$-manifolds, the picture is now nearly complete.
By a result of Kleiner~\cite[Theorem~2]{Kle-I} (generalized in~\cite{Sch}), every such manifold $X$ satisfies a Euclidean isoperimetric inequality for domains with the sharp constant $1/(6\sqrt{\pi})$. 
Furthermore, by~\cite[Proposition~7.1]{DLPS}, an inequality with a strictly smaller constant for domains of large volume implies that $X$ has asymptotic rank at most $2$; 
then a $\delta$-isoperimetric inequality holds for all $\del > 0$. 
In addition, if $X$ covers a compact manifold, then a linear inequality holds. 
In fact, if $M$ is a compact non-positively curved $n$-manifold, then it is known that either $M$ is flat
or the fundamental group of $M$ is non-amenable and hence satisfies a combinatorial linear isoperimetric inequality (see~\cite[6.8, 6.14, 6.26]{Gro-MS} and~\cite[Corollary C]{AdaB}); consequently, the universal cover of $M$ admits a linear isoperimetric inequality for domains by~\cite[Theorem~6.19]{Gro-MS}. 

We now comment on the proof strategy and the further structure of the paper. The overall approach is similar as in~\cite{DLPS} in that it uses a decomposition process and a linear isoperimetric inequality for cycles of a specific type. However, the concrete arguments deviate substantially, and the present paper may be read independently of~\cite{DLPS}.

In Section~\ref{sect:min-simplices}, we will discuss {\em minimizing $k$-simplices} spanned by the geodesic $1$-skeleton determined by $k+1$ points in $X$. 
These are defined iteratively by filling in area minimizing faces (integral currents) in dimensions $2,\ldots,k$.
If the asymptotic rank $\nu$ of $X$ is at most $2$, then it follows from results in~\cite{KleL, GolL, Gol-Diss} that minimizing $k$-simplices with $k > \nu$ are `slim' and satisfy a linear isoperimetric inequality (see Theorem~\ref{Thm:slim-simplices} for details). 
This hinges on a uniform upper density bound for minimizing $2$-simplices (Lemma~\ref{Lem:min-triangles}). 
The validity of an analogous bound for higher-dimensional simplices is unknown, and this is the only reason for the restriction to asymptotic rank $2$ in Theorem~\ref{Thm:main}.
As a consequence of the linear isoperimetric inequality, {\em piecewise minimizing} $k$-cycles with $k \ge \nu$, composed of finitely many minimizing simplices, admit efficient fillings of conical type (see Proposition~\ref{Pro:p-filling}).

In Section~\ref{sect:approx}, we prove that a general $k$-cycle $T$ can be approximated by a piecewise minimizing cycle $P$ with vertices in $T$, in dependence of a scale $s > 0$, such that the difference $T - P$ admits a decomposition into finitely many `small' cycles $Z_i$, and both the mass of $P$ and the total mass of the $Z_i$ are bounded by a constant times the mass of $T$ (see Theorem~\ref{Thm:approx}).
This result is independent of the asymptotic rank, but requires suitable controlled coverings of $T$, whose existence is guaranteed, for large scales $s$, if the asymptotic Nagata dimension of $X$ is finite.

Section~\ref{sect:main} concludes the proof of Theorem~\ref{Thm:main}. 
The same reduction as in the proof of Wenger's Euclidean isoperimetric inequality shows that it suffices to consider `round' $k$-cycles $R$, whose diameter is bounded by a certain constant times $\M(R)^{1/k}$. 
The approximation result from Section~\ref{sect:approx} is applied iteratively, for a finite decreasing sequence of scales.
In each step, the respective piecewise minimizing cycles are filled by the result from Section~\ref{sect:min-simplices},
and the remaining cycles are further decomposed using the next smaller scale, except in the last step, when they are small enough to admit efficient cone fillings.

It remains an open question whether a linear isoperimetric filling inequality holds under the assumptions of Theorem~\ref{Thm:main}, or whether the condition on the asymptotic Nagata dimension can be discarded altogether. 


\section{Asymptotic Nagata dimension} \label{sect:nagata}

Let $X$ be a metric space. 
A family $\cB = (B_i)_{i \in I}$ of subsets of $X$ is called {\em $D$-bounded} if every $B_i$ has diameter at most $D$, and $\cB$ has {\em $s$-multiplicity at most $m$} if every set of diameter $\le s$ in $X$ meets no more than $m$ members of the family. 
The {\em asymptotic dimension} of $X$ is the infimum of all integers $n$ such that for every $s > 0$ there exists a $D(s)$-bounded covering of $X$ with $s$-multiplicity at most $n+1$ for some $D(s) < \infty$. 
Note that there is no effect on small scales as it is not required that $D(s)$ tends to zero with~$s$.
This invariant was introduced (and called $\asdim_+$) by Gromov in~\cite[1.E]{Gro-AI}.
He also observed that in many cases the function $D$ can be chosen linear in the large.
Accordingly, $X$ has {\em linearly controlled asymptotic dimension} at most $n$ if there exist constants $r,c > 0$ such that for every $s > r$, $X$ admits a $cs$-bounded covering with $s$-multiplicity at most $n+1$. 
Requiring the latter condition for all scales $s > 0$, one arrives at the notion of {\em Nagata dimension}, introduced earlier by Assouad in~\cite{Ass} and put forward in~\cite{LanS-ND}. 
For this reason, the linearly controlled asymptotic dimension is also known as {\em asymptotic Assouad--Nagata dimension}. 
For brevity, we will use the term {\em \asNdim}.
These definitions have various equivalent reformulations (see, for example,~\cite[Proposition~1.7]{DraS} for the~\asNdim).
In particular, the dimensions are not affected if the covers are taken open or the `test set' of diameter at most~$s$ is replaced with an open or closed ball of radius~$s$.

The literature on these invariants is vast, and we refer to~\cite{BelD, BuyS, Roe} for some surveys in different directions. 
Notice that the \asNdim\ is a quasi-isometry invariant, whereas the asymptotic dimension is a quasi-isometry and coarse invariant, and the Nagata dimension is a bi-Lipschitz, quasi-symmetry, and quasi-M\"obius invariant; see~\cite[Theorem~1.2]{LanS-ND} and~\cite{Xie}.
We list some relevant facts in the context of non-positive curvature.

\begin{enumerate}
\item[(I)] {\em Hadamard manifolds.} 
Every {\em planar} geodesic metric space, admitting an injective continuous map to $\R^2$, has Nagata dimension at most~$2$, and every $3$-dimensional Hadamard manifold has Nagata dimension~$3$ by~\cite[Theorems~2 and~3]{JorL} (based on \cite{FujP}). 
For $n \ge 4$, Hadamard $n$-manifolds of pinched negative curvature are known to have Nagata dimension $n$, and homogeneous Hadamard manifolds have finite Nagata dimension~\cite[Sect.~3]{LanS-ND}. 
It is an open question whether every Hadamard manifold has finite Nagata dimension. 
By~\cite[Corollary~1.8]{LanS-ND}, every $\CAT(0)$ space with finite Nagata dimension is an absolute Lipschitz retract. 
\item[(II)] {\em Products and Euclidean buildings.} 
For each of the three notions, asymptotic, asymptotic Nagata, and Nagata dimension, the dimension of a product $X \times Y$ of metric spaces is at most the sum of the dimensions of the factors~\cite[Theorem~2.5]{Bro+}.
Products of $n$ non-trivial metric \mbox{($\R$-)}trees as well as Euclidean buildings of rank $n$ have Nagata dimension $n$~\cite[Sect.~3]{LanS-ND}.
\item[(III)] {\em Gromov hyperbolic spaces.}
Let $X$ be a geodesic Gromov hyperbolic space. Then $X$ has finite \asNdim\ provided that for a single sufficiently large scale $s > 0$, depending on the hyperbolicity constant, $X$ has a uniformly bounded covering with finite $s$-multiplicity. 
This is satisfied if $X$ is doubling at some scale (unrelated to the hyperbolicity constant); see~\cite[Theorem~3.5]{LanS-ND} and the comment thereafter. 
Furthermore, it follows from~\cite[Theorem~1.1]{Buy} (see also~\cite[Theorem~12.1.1]{BuyS} and~\cite[Theorem~1.5]{Leb}) that if $X$ is visual, then the \asNdim\ of $X$ is less than or equal to the Nagata dimension of the boundary $\d_\infty X$, equipped with any visual metric, plus $1$.
\item[(IV)] {\em $\CAT(0)$ cube complexes.}
Let $X$ be a finite-dimensional (but not necessarily locally finite) $\CAT(0)$ cube complex. 
Wright~\cite[Theorem~4.10]{Wri} showed that the asymptotic dimension of $X$ is at most the dimension of~$X$. 
Subsequently, another proof of the finiteness of the asymptotic dimension was given in~\cite{Arz+}, and the argument actually provides a linear control, thus showing that the \asNdim\ of $X$ is finite; 
see the proof of Corollary~3.3 and the relation $S(x,k,l) \sub B(x,Ml)$ on p.~515 in that paper ($l$\/ is the scale and $Ml$\/ the diameter bound).
\item[(V)]
{\em Asymptotic cones.} 
The Nagata dimension, and hence also the topological dimension, of any asymptotic cone of a metric space $X$ is bounded above by the \asNdim\ of $X$~\cite[Corollary~4.3]{DydH}.
\end{enumerate}

For our main result, the finiteness of the \asNdim\ will be used in two different ways to approximate $k$-cycles by piecewise minimizing ones; 
see the proofs of Proposition~\ref{Pro:homotopy} and Theorem~\ref{Thm:approx}. 
The latter employs the following standard fact (compare~\cite[Proposition~6.1]{BasWY} and~\cite[Lemma~5.2]{MeiW}). We provide the argument for convenience.

\begin{Lem} \label{Lem:complex}
Suppose that $X$ is a $\CAT(0)$ space, and $Y \sub X$ is a
subset with a finite $cs$-bounded covering with $s$-multiplicity at most $n + 1$,
for some integer $n$ and positive constants $c$ and $s$. Then there exist
\begin{enumerate}
\item[\rm (1)]
a finite simplicial complex $\Sig$ of dimension at most $n$, metrized as a
subcomplex of some simplex of edge length $s$ in a Euclidean space;
\item[\rm (2)]
a constant $L > 0$ depending only on $n$ and $c$, and an $L$-Lipschitz
map $\psi \colon Y \to \Sig$;
\item[\rm (3)]  
a Lipschitz map $\phi \colon \Sig \to X$ with $\phi(\Sig^{(0)}) \sub Y$ that is
$L$-Lipschitz on every simplex of\/ $\Sig$, such that $\phi \circ \psi$ is
$L$-Lipschitz and satisfies $d(x,\phi \circ \psi(x)) \le Ls$ for all $x \in Y$.
\end{enumerate}
\end{Lem}

\begin{proof}
Let $(B_i)_{i=1}^N$ be a finite $cs$-bounded covering of $Y$ with $s$-multiplicity at most $n + 1$.
Suppose that $\es \ne B_i \sub Y$, and define
\[
  \tau = (\tau_1,\ldots,\tau_N) \colon Y \to \R^N, \quad
  \tau_i(x) := \max\{s - 2\,d(x,B_i), 0\}.
\]
Note that $\bar\tau(x) := \sum_{i=1}^N \tau_i(x) \ge s$ for all $x \in Y$,
and there are always at most $n+1$ non-zero summands. It follows that
$\psi := s\,\tau/(\sqrt{2}\,\bar\tau)$ maps $Y$ into the standard
$(N-1)$-simplex $\Del \sub \R^N$ of edge length $s$,
and the minimal subcomplex $\Sig \sub \Delta$ containing $\psi(Y)$ has
dimension at most $n$. Given $x,y \in Y$, there are at most $2n + 2$
indices $i$ such that $\tau_i(x) + \tau_i(y) > 0$. For each of them,
\begin{align*}
  \biggl| \frac{s\,\tau_i(x)}{\sqrt{2}\,\bar\tau(x)}
  - \frac{s\,\tau_i(y)}{\sqrt{2}\,\bar\tau(y)} \biggr|
  &\le \frac{s\,|\tau_i(x) - \tau_i(y)|}{\sqrt{2}\,\bar\tau(x)} 
    + \frac{s\,\tau_i(y)\,|\bar\tau(y) - \bar\tau(x)|}
    {\sqrt{2}\,\bar\tau(x)\,\bar\tau(y)} \\
  &\le \sqrt{2}\,d(x,y)
    + \frac{\tau_i(y)}{\bar\tau(y)} \cdot (2n+2) \cdot \sqrt{2}\,d(x,y), 
\end{align*}
because $s \le \bar\tau(x)$ and the $\tau_i$ are $2$-Lipschitz.
Taking the sum, we conclude that $\psi$ is $4\sqrt{2}\,(n + 1)$-Lipschitz.

The reverse map $\phi \colon \Sig \to X$ is constructed by induction on the
skeleta of $\Sig$. For $i = 1,\ldots,N$, choose $z_i \in B_i$ and put
$\phi(s\,e_i) := z_i$. For $\es \ne I \sub \{1,\ldots,N\}$, let $\Del_I$
denote the subsimplex of $\Delta$ with vertex set $\{s\,e_i: i \in I\}$.
Notice that $\Del_I$ belongs to $\Sig$ if and only if there exists
a point $x \in Y$ with $\tau_i(x) > 0$ for all $i \in I$.
Then $d(x,B_i) < s/2$ for all $i \in I$, and the set of all
$z_i = \phi(s\,e_i)$ with $i \in I$ has diameter at most $(1 + 2c)s$.
For every $1$-simplex $\Del_I \sub \Sig^{(1)}$ with $I = \{i,j\}$,
define $\phi|_{\Del_I}$ as the constant speed geodesic connecting $z_i$ and
$z_j$, so that $\phi|_{\Del_I}$ is Lipschitz with constant $L_1 := 1 + 2c$.
Suppose now that $\phi$ is defined on $\Sig^{(i)}$ and $L_i$-Lipschitz
on every $i$-simplex. If $\Del_I \sub \Sig^{(i+1)}$ is an $(i+1)$-simplex,
then $\phi|_{\partial \Del_I}$ is Lipschitz with constant $L_{i+1} := c_iL_i$
for some $c_i > 1$ depending only on $i$. Hence, by the generalized Kirszbraun
theorem~\cite[Theorem~A]{LanS-KT},
$\phi|_{\partial \Del_I}$ extends to an $L_{i+1}$-Lipschitz map
from $\Del_I$ into~$X$. This procedure eventually yields a map
$\phi \colon \Sig \to X$ that is $L_k$-Lipschitz on every
simplex of $\Sig$, where $k \le n$ is the dimension of~$\Sig$.
Furthermore, $\phi$ is Lipschitz (with a constant that is possibly not
bounded in terms of $n$ and $c$), because $\phi$ is locally Lipschitz and
$\Sig$ is compact.

Now consider the composition $f := \phi \circ \psi \colon Y \to X$.
Let $x \in Y$, and pick $i$ with $x \in B_i$. Then $\tau_i(x) = s$,
in particular $\psi(x)$ belongs to a simplex of $\Sig$ with vertex $s\,e_i$,
and
\[
  d(x,f(x)) \le d(x,z_i) + d(\phi(s\,e_i),\phi(\psi(x)))
  \le (c + L_k)\,s.  
\]
If $y \in Y$ is another point, then either $\psi(x)$ and $\psi(y)$ lie
in a common simplex of $\Sig$, so that
\[
  d(f(x),f(y)) \le L_k\,d(\psi(x),\psi(y))
  \le 4\sqrt{2}\,(n+1)\,L_k\,d(x,y),
\]
or $\tau_i(y) = 0$, in which case $d(x,y) \ge d(y,B_i) \ge s/2$ and
\begin{align*}
  d(f(x),f(y)) &\le d(x,y) + d(x,f(x)) + d(y,f(y)) \\
  &\le d(x,y) + 2(c + L_k)\,s \\
  &\le (1 + 4c + 4L_k)\,d(x,y).
\end{align*}
This completes the proof.
\end{proof}


\section{Metric currents} \label{sect:currents}

We now review the basic notions of the theory of currents in a complete metric space $X$, introduced by Ambrosio--Kirchheim in~\cite{AmbK}. 
We restrict ourselves to currents with compact support, the main reason being that the existence result for minimizing fillings of integral cycles in $\CAT(0)$ spaces (Theorem~\ref{Thm:min-fill}), which plays a central role in this paper, relies on this assumption. 
It also allows to proceed similarly as in~\cite{Lan-LC}, without including a finite mass condition from the beginning, in analogy to the classical theory.
The same setup is used in~\cite{Zue}, which provides further details.

We write $\Lip(X)$ for the vector space of real valued Lipschitz functions on~$X$, equipped with the notion of pointwise convergence with uniformly bounded Lipschitz constants, which implies uniform convergence on compact subsets. 
For $k \ge 0$, a general {\em $k$-dimensional current $T$} in $X$ with support in some compact set $K \sub X$ is a $(k+1)$-linear functional $T \colon \Lip(X)^{k+1} \to \R$ that is jointly sequentially continuous and satisfies $T(f_0,\ldots,f_k) = 0$ whenever $\spt(f_0) \cap K = \es$ or one of the functions $f_1,\ldots,f_k$ is constant on the $\del$-neighborhood of $\spt(f_0)$ for some $\del > 0$. 
There then exists a smallest such compact set $K$, the {\em support}\/ $\spt(T)$ of $T$, and it follows by continuity that $T(f_0,\ldots,f_k) = 0$ if $f_0 = 0$ on $\spt(T)$ or if one of $f_1,\ldots,f_k$ is constant on $\spt(f_0)$ (see~\cite[Lemma~2.3]{Zue}).
We denote the vector space of $k$-currents with compact support in $X$ by $\cD_{k,\cs}(X)$. 

The {\em boundary} of $T \in \cD_{k,\cs}(X)$, for $k \ge 1$, is the $(k-1)$-current $\d T$ defined by
\[
\d T(f_0,\ldots,f_{k-1}) := T(1,f_0,\ldots,f_{k-1}).
\]
To see that indeed $\d T \in \cD_{k-1,\cs}(X)$, note that if, say, $f_1$ is constant on some $\del$-neighborhood $U$ of $\spt(f_0)$, then one can choose $\sig \in \Lip(X)$ such that $\sig \equiv 1$ on $\spt(f_0)$ and $\sig \equiv 0$ on $X \sm U$; then $T(1,f_0,\ldots,f_{k-1}) = T(\sig,f_0,\ldots,f_{k-1}) = 0$ because $f_0 \equiv 0$ on $\spt(1 - \sig)$ and $f_1$ is constant on $\spt(\sig)$.
Note further that $\spt(\d T) \sub \spt(T)$, and $\d(\d T) = 0$ for $k \ge 2$.

For a Lipschitz map $\phi \colon X \to Y$ to another complete metric space $Y$, the {\em pushforward}
$\phi_\#T \in \cD_{k,\cs}(Y)$ of $T \in \cD_{k,\cs}(X)$ is defined by
\[
\phi_\#T(g_0,\ldots,g_k) := T(g_0 \circ \phi,\ldots,g_k \circ \phi)
\]
for $(g_0,\ldots,g_k) \in \Lip(Y)^{k+1}$. If $k \ge 1$, then $\d(\phi_\#T) = \phi_\#(\d T)$.

The {\em mass}\/ $\M_V(T)$ of $T \in \cD_{k,\cs}(X)$ in an open set $V \sub X$ is defined as the supremum of $\sum_{\lam} T(f_{\lam,0},\dots,f_{\lam,k})$ over all finite families of tuples such that the sum $\sum_\lam |f_{\lam,0}|$ is at most $1$ and has support in $V$, and $f_{\lam,1},\dots,f_{\lam,k}$ are $1$-Lipschitz for all $\lam$.
The total mass $\M(T) := \M_X(T)$ is a norm on
\[
\M_{k,\cs}(X) := \{T \in \cD_{k,\cs}(X): \M(T) < \infty\}.
\] 
For an $L$-Lipschitz map $\phi \colon X \to Y$,
\[
\M(\phi_\#T) \le L^k \,\M(T).
\]
If $T \in \M_{k,\cs}(X)$, then the set function $\|T\| \colon 2^X \to [0,\infty)$ defined by
\[
\|T\|(A) := \inf\{\M_V(T): V \sub X \text{ open, } A \sub V\}
\]
is a Borel regular outer measure, $\|T\|(X \sm \spt(T)) = 0$, 
and if $f_i$ is $l_i$-Lipschitz for $i = 1,\ldots,k$, then
\[
|T(f_0,\ldots,f_k)| \le \int_X |f_0| \,d\|T\| \cdot l_1 \cdot \ldots \cdot l_k
\]
(see~\cite[Theorem~4.3]{Lan-LC} and~\cite[Lemma~2.4]{Zue}). 
This inequality implies further that $T$ extends canonically to tuples whose first entry is in $L^1(\|T\|)$. 
In particular, for a Borel set $B \sub X$ with characteristic function $\chi_B$,
\[
(T \on B)(f_0,\ldots,f_k) := T(\chi_B f_0,f_1,\ldots,f_k)
\]
defines a current $T \on B \in \M_{k,\cs}(X)$ with mass $\M(T \on B) = \|T\|(B)$.

As a basic example, every function $\theta \in L^1(\R^k)$ with compact support
induces a current $T = \bb{\theta} \in \M_{k,\cs}(\R^k)$, 
\[
\bb{\theta}(f_0,\ldots,f_k) := \int_{\R^k} \theta\,f_0 \det[\partial_j f_i]_{i,j=1}^k \,dx
\]
(the partial derivatives exist almost everywhere by Rademacher's theorem), 
with $d\|T\| = \theta\,dx$. 
For a bounded Borel set $D \sub \R^k$, we will write $\bb{D}$ instead of $\bb{\chi_D}$.  
Notice that $\bb{D}$ extends the classical (de Rham) current given by integration of smooth differential $k$-forms $f_0\,d f_1 \wedge \ldots \wedge d f_k$ on $D$.

Recall that a subset $E$ of a metric space is {\em countably $k$-rectifiable} if $E$ is the union of countably many Lipschitz images of bounded subsets of $\R^k$. 
A current $T \in \M_{k,\cs}(X)$ is {\em integer rectifiable} if $\|T\|$ is concentrated on some countably $k$-rectifiable Borel set, 
and if for every Borel set $B \sub X$ and every Lipschitz map $\phi \colon X \to \R^k$, the current $\phi_\#(T \on B)$ is of the form $\bb{\theta}$ for some integer valued $\theta \in L^1(\R^k)$. 
Then $\|T\|$ is absolutely continuous with respect to $k$-dimensional Hausdorff measure.
Pushforwards and restrictions to Borel sets of integer rectifiable currents are again integer rectifiable.

A $k$-current $T$ is called an {\em integral current}\/ if $T$ is integer rectifiable and, in case $k \ge 1$, $\d T$ has finite mass; then $\d T$ is integer rectifiable as well (see \cite[Theorem~8.6]{AmbK}). 
This gives a chain complex of abelian groups $\bI_{k,\cs}(X)$.
If $T \in \bI_{k,\cs}(X)$, $k \ge 1$, and $\rho \colon X \to \R$ is a $1$-Lipschitz function with sublevel sets $B_r := \{\rho \le r\}$, then for almost every $r \in \R$ the {\em slice}
\[
S_r := \d(T \on B_r) - (\d T) \on B_r
\]
is an integral $(k-1)$-current with support in $\spt(T) \cap \{\rho = r\}$.
Furthermore, for every subinterval $(a,b) \sub \R$, the coarea inequality
\[
\int_a^b \M(S_r) \,dr \le \|T\|(\{a < \rho < b\})
\]
holds (compare~\cite[Theorem~5.6]{AmbK} and~\cite[Theorem~6.2]{Lan-LC}).

\begin{Rem} \label{Rem:chains}
If $\sig \colon \Del \to X$ is a Lipschitz map of a $k$-simplex $\Del \sub \R^k$, then $\sig_\#\bb{\Del} \in \bI_{k,\cs}(X)$; thus every Lipschitz singular $k$-chain with integer coefficients determines an integral $k$-current. 
Conversely, an approximation result holds in case $X$ is a $\CAT(0)$ space with finite Nagata dimension:
if $T \in \bI_{k,\cs}(X)$, $k \ge 1$, then for every $\eps > 0$ there exists an integral current $T_\eps$ induced by a Lipschitz chain such that $\M(T-T_\eps) + \M(\d T - \d T_\eps) < \eps$; 
moreover, if $\d T$ is already given by a Lipschitz cycle, then $T_\eps$ can be chosen such that $\d T_\eps = \d T$
(compare~\cite[Corollary~1.5, Lemma~7.4]{BasWY} and~\cite[Theorem~1.3]{Gol}).
\end{Rem}

We now assume again that $X$ is a $\CAT(0)$ space and collect some crucial results. 
Given a cycle $T \in \bI_{k,\cs}(X)$, for $k \ge 1$, we call a current $S \in \bI_{k+1,\cs}(X)$ a {\em filling} of $T$ if $\d S = T$. 
A geodesic cone construction yields:

\begin{Pro}[coning inequality] \label{Pro:coning}
For every cycle $T \in \bI_{k,\cs}(X)$ with support in a
closed ball $B(x,r)$ there exists a filling
$S \in \bI_{k+1,\cs}(X)$ with
\[
  \M(S) \le \frac{r}{k+1}\,\M(T).
\]
\end{Pro}

See~\cite[Theorem~4.1]{Wen-FI}. 
An integral current is {\em minimizing} if it has minimal mass among all integral currents with the same boundary.

\begin{Cor}[monotonicity] \label{Cor:monotonicity}
If\/ $V \in \bI_{k+1,\cs}(X)$ is minimizing and $x \in \spt(V) \sm \spt(\d V)$, 
then
\[
  \om_{k+1} \le \frac{\|V\|(B(x,r))}{r^{k+1}}
  \le \frac{\|V\|(B(x,s))}{s^{k+1}}
\]
for $0 < r \le s \le d(x,\spt(\d V))$, where $\om_{k+1}$ denotes the Lebesgue
measure of the unit ball in $\R^{k+1}$.
\end{Cor}

See~\cite[Corollary 4.4]{Wen-FI}.

\begin{Thm}[Euclidean isoperimetric inequality] \label{Thm:eucl-isop}
For every cycle $T \in \bI_{k,\cs}(X)$ there exists a filling
$S \in \bI_{k+1,\cs}(X)$ with 
\[
  \M(S) \le \gam_k\,\M(T)^{1 + 1/k},
\]
where the constant $\gam_k$ depends only on $k$.
\end{Thm}

This is a particular case of~\cite[Theorem~1.2]{Wen-EII} (see also the remark thereafter regarding compact supports). 
The \emph{filling volume} of a cycle $T \in \bI_{k,\cs}(X)$, $k \ge 1$, is defined by
\[
\Fill(T) := \inf\{\M(S) : S \in \bI_{k+1,\cs}(X) \text{ and } \d S = T\}.
\]

\begin{Thm}[minimizing fillings] \label{Thm:min-fill}
For every cycle $T \in \bI_{k,\cs}(X)$ there exists a filling $S \in \bI_{k+1,\cs}(X)$ with
$\M(S) = \Fill(T)$.   
\end{Thm}

See~\cite[Theorem~1.6]{Wen-EII}. Notice that here $X$ is still a general (complete but not necessarily locally compact) $\CAT(0)$ space; the compactness of $\spt(T)$ suffices.


\section{Minimizing simplices} \label{sect:min-simplices}

This section is based in part on~\cite{GolL} and~\cite[Chapter~8]{Gol-Diss}.

Let again $X$ be a $\CAT(0)$ space. 
We define minimizing simplices inductively as follows.
For $x_0 \in X$, $\bb{x_0} \in \bI_{0,\cs}(X)$ is defined by
\[
  \bb{x_0}(f) := f(x_0).
\]
For $x_0,x_1 \in X$, $\bb{x_0,x_1} \in \bI_{1,\cs}(X)$ is the current induced
by the geodesic $\gam \colon [0,1] \to X$ from $x_0$ to $x_1$,
\[
  \bb{x_0,x_1}(f,g) := \int_0^1 (f \circ \gam)(g \circ \gam)' \,dt,  
\]
whose boundary is $\bb{x_1} - \bb{x_0}$. For $x_0,x_1,x_2 \in X$,
a {\em minimizing triangle} $\bb{x_0,x_1,x_2} \in \bI_{2,\cs}(X)$
is a minimizing filling of the cycle
$\bb{x_0,x_1} + \bb{x_1,x_2} + \bb{x_2,x_0}$.
In general, for $k \ge 2$, a {\em minimizing $k$-simplex}
$\bb{x_0,\ldots,x_k} \in \bI_{k,\cs}(X)$ with vertices 
$x_0,\ldots,x_k \in X$ is a minimizing filling of the cycle
\[
  \sum_{i=0}^k (-1)^i \,\bb{x_0,\ldots,x_{i-1},x_{i+1},\ldots,x_k},
\]
where each $\bb{x_0,\ldots,x_{i-1},x_{i+1},\ldots,x_k}$ is a minimizing 
$(k-1)$-simplex. 
This uses Theorem~\ref{Thm:min-fill}, which does not grant uniqueness. 
It is unclear if minimizing fillings of geodesic triangles in $\CAT(0)$ spaces 
are unique and if they are of the type of the disc. However, in the case
of $X = \R^n$, the construction provides the standard affine simplices
(compare \cite[p.~365]{Fed}).

\begin{Lem} \label{Lem:s-simplex}
Let $X$ be a $\CAT(0)$ space.
If $P = \bb{x_0,\ldots,x_k} \in \bI_{k,\cs}(X)$ is a minimizing $k$-simplex
with $\diam\{x_0,\ldots,x_k\} \le s$, then $\M(P) \le c_k s^k$
for some constant $c_k$ depending only on $k$.
\end{Lem}

\begin{proof}
For $k = 2$, one can take $c_2 = 1/2$ by Proposition~\ref{Pro:coning}.
The general case follows by induction on $k$ from Theorem~\ref{Thm:eucl-isop}.
\end{proof}  

For a current $T \in \bI_{k,\cs}(X)$ and constants $\theta,\rho \ge 0$, 
we say that $T$ has {\em $(\theta,\rho)$-controlled density} if
\[
  \frac{\|T\|(B(x,r))}{r^k} \le \theta
\]
for all $x \in X$ and $r > \rho$.

\begin{Lem} \label{Lem:min-triangles}
In a $\CAT(0)$ space $X$, every minimizing
$1$-simplex has $(2,0)$-controlled density, and every 
minimizing $2$-simplex has $(3\pi/2, 0)$-controlled density.
\end{Lem}

\begin{proof}
The result for $1$-simplices is clear. For $2$-simplices, this uses an extension trick from~\cite[Sect.~3.1, Theorem~81]{Sta} based on
Reshetnyak's gluing theorem for $\CAT(0)$ spaces~\cite[p.~347]{BriH}.
First, extend $X$ to a space $X'$ by attaching a half-line to each vertex of the simplex. 
Each edge of the simplex together with the two adjacent rays forms a geodesic line in $X'$. 
Now extend $X'$ further by gluing a Euclidean half-plane along each of these three lines, identifying the line with the boundary of the half-plane. 
By the gluing theorem, this yields a $\CAT(0)$ space $X''$ containing $X$ isometrically. 
It can be shown that the minimizing simplex together with the three added half-planes, properly oriented, forms a (locally) minimizing locally integral $2$-current $V$ in $X''$. 
Since $V$ has asymptotic density $3\pi/2$, for any base point $x \in X$, the result follows from the monotonicity formula, Corollary~\ref{Cor:monotonicity}. 
See~\cite[Lemma~8.4]{Gol-Diss} for more details.
\end{proof}

In contrast to Lemma~\ref{Lem:min-triangles}, \cite[Example~8.6]{Gol-Diss}
shows that the filling of a geodesic triangle given by the geodesic
cone from some vertex over the opposite side may fail to have uniformly 
bounded density.

We now turn to the assumption involving the rank. 
For a $\CAT(0)$ space~$X$, the {\em asymptotic rank\/} is defined as the supremum of all 
$m \ge 0$ for which there exist a sequence $0 < r_i \to \infty$ and subsets $S_i \sub X$ such that the
rescaled sets $(S_i,\frac{1}{r_i}d)$ converge in the Gromov--Hausdorff topology to the unit ball $B^m \sub \R^m$. As mentioned earlier, if $X$ is proper and cocompact, this agrees with the Euclidean rank; compare~\cite[Theorems~C and~7.1]{Kle} and \cite[Theorem~3.4]{Wen-AR}. 
The fact that the following results build on Lemma~\ref{Lem:min-triangles}
is the only reason for the restriction to asymptotic rank two in this paper.

\begin{Pro} \label{Pro:min-simplices}
Let $X$ be a $\CAT(0)$ space with asymptotic rank $\nu \in \{1,2\}$.
Then for all\/ $k > \nu$ and\/ $\theta > 0$ there exists a constant
$\rho \ge 0$ such that every minimizing $k$-simplex $P \in \bI_{k,\cs}(X)$ has
$(\theta,\rho)$-controlled density.
\end{Pro}

\begin{proof}
By Lemma~\ref{Lem:min-triangles} we can assume inductively that every facet of $P$ (of dimension $k - 1 \ge \nu$) has $(\theta',\rho')$-controlled density for some $\theta',\rho' \ge 0$. 
Then $\d P$ has $((k+1)\theta',\rho')$-controlled density, and it follows as in~\cite[Proposition~5.2]{GolL} (where $X$ is assumed to be proper) that $P$ has $(\theta,\rho)$-controlled density for some constant $\rho = \rho(X,k,\theta',\rho',\theta)$. 
The argument relies on the sub-Euclidean isoperimetric inequality~\cite[Theorem~1.2]{Wen-AR} and applies verbatim.
\end{proof}

We now have the following analogue of the slim triangle property in Gromov hyperbolic spaces.

\begin{Thm} \label{Thm:slim-simplices}
Let $X$ be a $\CAT(0)$ space with asymptotic rank $\nu \in \{1,2\}$,
and let $k > \nu$. Then there exists a constant $D_k \ge 0$ such that the
following holds for every minimizing $k$-simplex
$P = \bb{x_0,\ldots,x_k} \in \bI_{k,\cs}(X)$
with facets $P'_i = \bb{x_0,\ldots,x_{i-1},x_{i+1},\ldots,x_k}$:
\begin{enumerate}
\item[\rm (1)]
  $\M(P) \le D_k\,\M(\d P)$;
\item[\rm (2)] 
  $\spt(P)$ is within distance at most $D_k$ from $\spt(\d P)$;
\item[\rm (3)]
  for every $i \in \{0,\ldots,k\}$, $\spt(P'_i)$ lies in the
  closed $D_k$-neighborhood of\/ $\bigcup_{j \ne i} \spt(P'_j)$.
\end{enumerate}
\end{Thm}

\begin{proof}
By Lemma~\ref{Lem:min-triangles} and Proposition~\ref{Pro:min-simplices},
$\d P$ is a $(k-1)$-cycle with uniformly controlled density.
Hence, assertions~(1), (2), (3) follow from 
the implications (AR$_{k-1}$) $\Rightarrow$ (LII$_{k-1}$), (FR$_{k-1}$),
(ML$_{k-1}$) in~\cite[Theorem~1.1]{GolL}, the last of which was shown
earlier in~\cite[Theorem~5.1]{KleL}. The second implication holds
for every minimizing filling of the cycle by~\cite[Theorem~7.3]{GolL}.
In these references, $X$ is assumed to be a proper metric space satisfying
coning inequalities, mainly to guarantee the existence of minimizing fillings.
By Theorem~\ref{Thm:min-fill} and Proposition~\ref{Pro:coning}, the arguments 
apply to general $\CAT(0)$ spaces.
\end{proof}

As a consequence of inequality~(1), the mass of a minimizing $k$-simplex
with $k \ge \nu \in \{1,2\}$ behaves like a $\nu$-dimensional quantity:

\begin{Cor} \label{Cor:simplex-mass}
Let $X$ be a $\CAT(0)$ space with asymptotic rank $\nu \in \{1,2\}$,
and let $k \ge \nu$. Then there exists a constant $C_k$ such that every minimizing
$k$-simplex $P = \bb{x_0,\ldots,x_k}$ in $X$ with $x_1,\ldots,x_k \in B(x_0,r)$
satisfies $\M(P) \le C_k \,r^\nu$; furthermore, if $\nu = 2$ and
$\diam\{x_1,\ldots,x_k\} \le s$, then $\M(P) \le C_k \,r s$.
\end{Cor}

\begin{proof}
This follows by induction on $k$. We just prove the last assertion.
By Proposition~\ref{Pro:coning}, every minimizing
triangle $P' = \bb{x_0',x'_1,x'_2}$ in $X$ with
$x'_1,x'_2 \in B(x'_0,r)$ and $d(x'_1,x'_2) \le s$ has
mass $\M(P' ) \le rs/2$. Suppose now that $k > \nu = 2$ and the inequality 
$\M(P') \le C_{k-1} \,rs$ holds for all minimizing $(k-1)$-simplices
$P' = \bb{x'_0,\ldots,x'_{k-1}}$ in $X$ with $x'_1,\ldots,x'_{k-1} \in B(x'_0,r)$
and $\diam\{x'_1,\ldots,x'_{k-1}\} \le s$.
Then, for $P$ as in the statement, we have
\[
  \M(\d P) \le C_{k-1}(k \cdot rs + 2rs), 
\]
and it follows from the first assertion in
Theorem~\ref{Thm:slim-simplices} that $\M(P) \le C_k\,rs$ for
$C_k := (k+2)\,C_{k-1}D_k$.
\end{proof}

By a {\em piecewise minimizing}\/ current $P \in \bI_{k,\cs}(X)$ we mean
an integral linear combination
\[
  P = \sum_{i=1}^N m_i P_i, \quad 0 < m_1,\ldots,m_N \in \Z,
\]
of some minimizing $k$-simplices $P_1,\ldots,P_N$ (in particular, the boundary of a minimizing
simplex is piecewise minimizing in this sense).
For such a representation, we write
\[
  |P|_1 := \sum_{i=1}^N m_i
\]
and denote by $\mesh(P)$ the maximal edge length of all simplices $P_i$.
By Lemma~\ref{Lem:s-simplex},
\[
\M(P) \le c_k \mesh(P)^k\,|P|_1.
\] 
The {\em vertex set\/} of $P$ is the union of the vertex sets of the $P_i$.

\begin{Pro} \label{Pro:p-filling}
Let $X$ be a $\CAT(0)$ space with asymptotic rank $\nu \in \{1,2\}$,
and let $k \ge \nu$. Then every piecewise minimizing cycle $P \in \bI_{k,\cs}(X)$
with vertex set $V$ has a piecewise minimizing filling $Q \in \bI_{k+1,\cs}(X)$ with
\[
\M(Q) \le C_{k+1} \diam(V) \mesh(P)^{\nu - 1}\, |P|_1. 
\]
\end{Pro}

\begin{proof}
Let $P = \sum_{i=1}^N m_i P_i$, and let $z \in V$ be a fixed vertex. We consider the set of all minimizing subsimplices of $P$ of any dimension $j \in \{0,\ldots,k\}$, that is, those arising in the inductive construction of the $k$-simplices $P_i$. 
We assign to each such $j$-simplex $P'$ a minimizing $(j+1)$-simplex $Q(P')$ with the additional vertex $z$, and to $-P'$ the oppositely oriented simplex $-Q(P')$, by the following inductive procedure.
For $j = 0$ and $x \in V$, we take the geodesic segment from $z$ to $x$ and put $Q(\bb{x}) := \bb{z,x}$. 
For $j = 1,\ldots,k$, if $\d P' = P_0'' + \ldots + P_j''$, then $P' - Q(P_0'') - \ldots - Q(P_j'')$ is a cycle and we let $Q(P')$ be a minimizing filling.
Finally, we put $Q_i := Q(P_i)$. As $P$ is closed,
$Q := \sum_{i=1}^N m_i Q_i$ is a filling of $P$.
By Corollary~\ref{Cor:simplex-mass}, 
\[
\M(Q_i) \le C_{k+1} \diam(V) \mesh(P)^{\nu - 1},
\]
and since $|Q|_1 = |P|_1$, the result follows. 
\end{proof}


\section{Piecewise minimizing approximation} \label{sect:approx}

Our next goal is to prove Theorem~\ref{Thm:approx} below.
The results in this section are independent of the asymptotic rank.

We start with the following result for displacements of cycles. 
Here and subsequently, we will simply write $\diam(R)$ rather than $\diam(\spt(R))$ 
for the diameter of (the support of) a current $R \in \bI_{k,\cs}(X)$. 

\begin{Pro} \label{Pro:homotopy}
Let $X$ be a $\CAT(0)$ space.
Suppose that $T \in \bI_{k,\cs}(X)$ is a cycle of dimension $k \ge 1$,
and $(B_i)_{i=1}^J$ is a finite $cs$-bounded covering of $\spt(T)$
with $s$-multiplicity at most $n+1$, for some $n,c,s > 0$.
Suppose further that $f \colon \spt(T) \to X$ is an
$L$-Lipschitz map such that $d(x,f(x)) \le Ls$ for all $x \in \spt(T)$.
Then, for some constant $A$ depending only on $k,n,c$ and $L$,
there exist cycles $R_1,\ldots,R_J \in \bI_{k,\cs}(X)$
with $\diam(R_i) \le A s$ such that 
\[
  T - f_\#T = \sum_{i=1}^J R_i \quad \text{and} \quad 
  \sum_{i=1}^J \M(R_i) \le A\,\M(T).
\]
\end{Pro}

\begin{proof}  
We assume that the $B_i$ are non-empty subsets of $\spt(T)$.
Let $\rho_i \colon X \to \R$ be the distance function to $B_i$, and let
\[
  W_i := \{x \in X: \rho_i(x) < s/2\}.
\]
Note that every $x \in X$ belongs to at most $n+1$ members of
$(W_i)_{i=1}^J$. We claim that for $j = 1,\ldots,J$ there exist
$T_j,U_j \in \bI_{k,\cs}(X)$ and decompositions
\[
  T = \sum_{i=1}^j T_i + U_j 
\]
such that $\M(T) = \sum_{i=1}^j \M(T_i) + \M(U_j)$  and
\[
  \sum_{i=1}^j \M(\d T_i) + \M(\d U_j) 
  \le \sum_{i=1}^j 4s^{-1}\|T\|(W_i).
\]
For $j = 1$, by the coarea inequality for slices,
there exists $r_1 \in (0,s/2)$ such that  
$T_1 := T \on \{\rho_1 \le r_1\}$ and $U_1 := T - T_1$
satisfy
\[
\M(\d U_1) = \M(\d T_1) \le 2s^{-1}\|T\|(W_1).
\]
Now let $j \in \{2,\ldots,J\}$. Given the decomposition for $j - 1$,
we split $U_{j-1}$ as $T_j + U_j$ such that
$T_j := U_{j-1} \on \{\rho_j \le r_j\}$ for some $r_j \in (0,s/2)$ and
the slice
\[
  S_j := \d T_j - (\d U_{j-1}) \on \{\rho_j \le r_j\}
  = (\d U_{j-1}) \on \{\rho_j > r_j\} - \d U_j
\]  
satisfies $\M(S_j) \le 2s^{-1} \|T\|(W_j)$. Then
\begin{align*}
  \M(\d T_j) + \M(\d U_j)
  &\le \|\d U_{j-1}\|(\{\rho_j \le r_j\}) + \M(S_j) \\
  &\qquad + \|\d U_{j-1}\|(\{\rho_j > r_j\}) + \M(S_j) \\
  &\le \M(\d U_{j-1}) + 4s^{-1} \|T\|(W_j).
\end{align*}
This yields the claim. Note that $\spt(T_j) \sub \{\rho_j \le r_j\}$ and
$\spt(U_j) \sub \spt(T) \sm \bigcup_{i=1}^j B_i$ for $j = 1,\ldots,J$,
in particular $\diam(T_j) \le (c+1)s$ and, in fact, $U_J = 0$.
Thus, we have decomposed $T$ into pieces $T_1,\dots,T_J$, and 
since $(W_i)_{i=1}^J$ is a covering of $\spt(T)$ with multiplicity at most $n + 1$, 
it follows that
\[
  \sum_{i=1}^J \M(\d T_i) \le \sum_{i=1}^J 4s^{-1}\|T\|(W_i) 
  \le 4(n+1)s^{-1} \,\M(T).
\]
For every $i$, the geodesic homotopy $h \colon [0,1] \times X \to X$
from $\id_X$ to the given map $f$ provides
a `cylinder' $Z_i \in \bI_{k,\cs}(X)$ with
$\d Z_i = f_\#(\d T_i) - \d T_i$ and
\[
  \M(Z_i) \le k L^k s \,\M(\d T_i)
\]
(see~\cite[Sect.~2.3]{Wen-EII}). The cycles $R_i := T_i + Z_i - f_\# T_i$ satisfy
\[
  \sum_{i=1}^J \M(R_i) \le \bigl( 1 + 4(n+1)k L^k + L^k \bigr) \,\M(T)
\]
and $\diam(R_i) \le (c+1+2L)\,s$. Note that $\sum_{i=1}^J Z_i = 0$
since $\sum_{i=1}^J \d T_i = \d T = 0$. This gives the result.
\end{proof}  

We will also use the following basic building block of the
deformation theorem in piecewise Euclidean simplicial complexes,
see~\cite[Appendix~A]{BasWY}.

\begin{Lem} \label{Lem:defo}
Let $\Sig$ be a complete metric space, and suppose that $\om \sub \Sig$ is an open subset isometric to a regular open Euclidean $l$-simplex of edge length $s > 0$. 
Then for every $T \in \bI_{k,\cs}(\Sig)$ with $1 \le k < l$ and $\spt(\d T) \cap \om = \es$ there exists a cycle $Z \in \bI_{k,\cs}(\Sig)$ with support in the closed simplex $\ol\om$ such that $\spt(T - Z) \cap \om = \es$ and $\M(Z) \le K\,\|T\|(\om)$ for some constant $K$ depending only on $n$.
\end{Lem}

\begin{proof}
This is shown as in the proof of~\cite[Proposition~A.4]{BasWY}.
The radial projection $\Sig \sm \{z\} \to \Sig \sm \om$ from a suitably chosen center $z \in \om$ maps $T$ to a current $T' \in \bI_{k,\cs}(\Sig)$ such that $Z := T - T'$ is the desired cycle. 
\end{proof}

We now have the following approximation result.

\begin{Thm} \label{Thm:approx}
Let $X$ be a $\CAT(0)$ space.
Suppose that $T \in \bI_{k,\cs}(X)$ is a cycle of dimension $k \ge 1$ whose
support has a finite $cs$-bounded covering with $s$-multiplicity at most $n+1$, for some $n,c,s > 0$. 
Then, for some constant $B > 0$ depending only on $k,n,c$,
there exist a piecewise minimizing cycle $P \in \bI_{k,\cs}(X)$ with vertices
in $\spt(T)$ and\/ $\mesh(P) \le Bs$, and a finite collection of
cycles $Z_1,\ldots,Z_N \in \bI_{k,\cs}(X)$ with $\diam(Z_i) \le Bs$
such that
\[
T - P = \sum_{i=1}^N Z_i, \quad 
|P|_1 \le B s^{-k}\,\M(T), \quad \text{and} \quad
\sum_{i=1}^N \M(Z_i) \le B\,\M(T).
\]
Furthermore, there exists an $S \in \bI_{k+1,\cs}(X)$ with $\d S = T - P$
and $\M(S) \le B s\,\M(T)$.
\end{Thm}

It follows from Lemma~\ref{Lem:s-simplex} that $\M(P) \le c_k(Bs)^k\,|P|_1 \le c_k B^{k+1}\,\M(T)$.

\begin{proof}
Choose a simplicial complex $\Sig$ and Lipschitz maps
$\psi \colon \spt(T) \to \Sig$ and $\phi \colon \Sig \to X$ as in
Lemma~\ref{Lem:complex}. Then $T' := \psi_\#T \in \bI_{k,\cs}(\Sig)$ satisfies
\[
  \M(T') \le L^k\,\M(T).
\]
Let $l \le n$ be the maximal dimension for which there exists a relatively
open $l$-simplex of $\Sig$ on which $\|T'\|$ is non-zero.
Writing $T^{\,l} := T'$, and assuming that $l > k$,
we let $\Om^{\,l}$ be the set of all such simplices.
We apply Lemma~\ref{Lem:defo} to the $l$-skeleton $\Sig^{(l)}$ and every $\om \in \Om^{\,l}$
to obtain cycles $Z'_\om \in \bI_{k,\cs}(\Sig)$
with support in $\ol\om$ such that 
$\|T^{\,l} - Z'_\om\|(\om) = 0$ and
\[
  \M(Z'_\om) \le K\,\|T^{\,l}\|(\om).
\]
Then $T^{\,l-1} := T^{\,l} - \sum_{\om \in \Om^{\,l}} Z'_\om$ is an integral
cycle with support in $\Sig^{(l-1)}$ and mass
\[
  \M(T^{\,l-1}) \le (1 + K)\,\M(T^{\,l}).
\]
Next, in case $l - 1 > k$, we let $\Om^{\,l-1}$ be the set of all open
$(l-1)$-simplices of $\Sig$ on which $\|T^{\,l-1}\|$ is non-zero, then we
repeat the above argument with $T^{\,l-1}$ and $\Om^{\,l-1}$ in place
of $T^{\,l}$ and $\Om^{\,l}$. Continuing in this manner, we eventually
get a cycle $P' := T^{\,k} \in \bI_{k,\cs}(X)$ with support in $\Sig^{(k)}$
and a decomposition
\[
  T' = P' + \sum_{\om \in \Om} Z'_\om,
\]
where $\Om = \Om^{k+1} \cup \ldots \cup \Om^l$, every $Z'_\om$
is a cycle with support in $\ol\om$, and
\[
  \M(P') \le K'\,\M(T'), \quad \sum_{\om \in \Om} \M(Z'_\om) \le K'\,\M(T')
\]
for some constant $K'$ depending only on $n$.
As an integral $k$-cycle with support in the $k$-skeleton,
$P'$ is in fact simplicial, $P' \in \cP_k(\Sig)$. That is,
\[
  P' = \sum_{i=1}^F m_i\bb{\sig_i}
\]
for some positive integers $m_i$ and oriented $k$-simplices $\sig_i$ of $\Sig$,
viewed as integral currents $\bb{\sig_i}$. Denoting by 
$\mu_k$ the mass of a Euclidean $k$-simplex of edge length $1$, we get
\[
  |P'|_1 := \sum_{i=1}^F m_i = \frac1{\mu_ks^k}\,\M(P')
  \le \frac{K' L^k}{\mu_k}\,\frac{\M(T)}{s^k} = K'' s^{-k} \,\M(T).
\]
Recall that $\phi \colon \Sig \to X$ is $L$-Lipschitz on
every simplex of $\Sig$. Put $Z_\om := \phi_\#(Z'_\om)$
for every $\om \in \Om$. Then $\diam(Z_\om) \le Ls$ and  
\[
  \sum_{\om \in \Om} \M(Z_\om) \le \sum_{\om \in \Om} L^k \,\M(Z'_\om)
  \le K'L^{2k} \,\M(T).
\]

Now we proceed similarly as in the proof of~\cite[Theorem~1.4]{BasWY}.
We successively choose homomorphisms 
\[
\Lam_j \colon \cP_j(\Sig) \to \bI_{j,\cs}(X), \quad j = 0,\ldots,k, 
\]
such that $\Lam_0(\bb{v}) = \phi_\#(\bb{v}) = \bb{\phi(v)}$ for every 
vertex $v$ of $\Sig$ and, if $\Lam_{j-1}$ is defined, such that $\Lam_j$ maps
every $j$-simplex $\bb{\sig}$ to a minimizing $j$-simplex in $X$ with boundary
\[
\d\,\Lam_j(\bb{\sig}) = \Lam_{j-1}(\d\,\bb{\sig}).
\]
Since $\phi$ is $L$-Lipschitz on simplices, it follows
from Lemma~\ref{Lem:s-simplex} that 
\[
\M(\Lam_j(\bb{\sig})) \le c_j(Ls)^j.
\]
Now $P := \Lam_k(P') \in \bI_{k,\cs}(X)$ is a piecewise minimizing cycle 
with vertices in $\spt(T)$, $\mesh(P) \le Ls$, and 
\[
|P|_1 = |P'|_1 \le K''s^{-k}\,\M(T),
\]
as desired. Next, we define homomorphisms
\[
\Gam_j \colon \cP_j(\Sig) \to \bI_{j+1,\cs}(X), \quad j = 0,\ldots,k,
\]
such that $\Gam_0 = 0$ and, if $\Gam_{j-1}$ is defined, such that $\Gam_j$ 
maps every $j$-simplex $\bb\sig$ to a minimizing integral 
$(j+1)$-current with boundary
\[
\d\,\Gam_j(\bb\sig) 
= \phi_\#(\bb\sig) - \Lam_j(\bb\sig) - \Gam_{j-1}(\d\,\bb\sig).
\]
It follows by induction from Theorem~\ref{Thm:eucl-isop} that 
\[
  \M(\d\,\Gam_j(\bb\sig) \le c'_j(Ls)^j \quad \text{and} \quad
  \M(\Gam_j(\bb\sig)) \le \gam_jc'_j(Ls)^{j+1}
\]
for some constants $c'_j$, as
\[
  \M(\d\,\Gam_j(\bb\sig)) \le L^j \mu_js^j + c_j(Ls)^j
  + (j+1)\,\gam_{j-1}c'_{j-1}(Ls)^j.
\]
Define $Z_i := \d\,\Gam_k(m_i\bb{\sig_i})$ for $i = 1,\ldots,F$. Then
\[
  \sum_{i=1}^F \M(Z_i) \le \sum_{i=1}^F m_i\,c'_k(Ls)^k
  \le c'_k K''L^k\,\M(T).
\]
Since $\d P' = 0$ and $P' = T' - \sum_{\om \in \Om} Z'_\om$, we have
\[
  \sum_{i=1}^F Z_i = \d\,\Gam_k(P') = \phi_\#(P') - \Lam_k(P')
  = f_\#T - \sum_{\om \in \Om} Z_\om - P,
\]
where $f := \phi \circ \psi$.
Finally, using Proposition~\ref{Pro:homotopy}, we get a decomposition
\[
  T - P = \sum_{i=1}^J R_i + \sum_{\om \in \Om} Z_\om +
  \sum_{i=1}^F Z_i,
\]
and the desired result follows upon relabeling the cycles. 

As for the last assertion of the theorem, a suitable $S \in \bI_{k+1,\cs}(X)$ with $\d S = T - P = \sum_{i=1}^N Z_i$ 
is obtained by applying the coning inequality (Proposition~\ref{Pro:coning}) 
to each $Z_i$.
\end{proof}  


\section{Proof of main result} \label{sect:main}

The proof of our main result relies crucially on Proposition~\ref{Pro:p-filling}, 
Theorem~\ref{Thm:approx}, and the following central element of 
Theorem~\ref{Thm:eucl-isop}; see~\cite[Proposition~3.1]{Wen-EII}.

\begin{Pro} \label{Pro:decomp}
Let $X$ be a $\CAT(0)$ space, and let $k \ge 1$.
There exist constants $\beta > 0$ and $\eps,\lam \in (0,1)$,
depending only on $k$, such that every cycle $T \in \bI_{k,\cs}(X)$ admits a
decomposition $T = T' + \sum_{i=1}^N R_i$ into finitely many cycles
$T',R_1,\ldots,R_N \in \bI_{k,\cs}(X)$ with
\begin{enumerate}
\item[\rm (1)] $\diam(R_i) \le \beta\,\M(R_i)^{1/k}$ for $i = 1,\ldots,N$,
\item[\rm (2)] $\M(T') \le (1-\eps)\,\M(T)$, and 
\item[\rm (3)] $\sum_{i=1}^N \M(R_i) \le (1 + \lam)\,\M(T)$.
\end{enumerate}
\end{Pro}

We restate Theorem~\ref{Thm:main} as follows.

\begin{Thm}
Let $X$ be a $\CAT(0)$ space with asymptotic rank $\nu \in \{1,2\}$ and finite
asymptotic Nagata dimension. Then for every integer $k \ge \nu$ and every $\del \in (0,1/k)$ there 
exists a constant $C > 0$ such that every cycle $T \in \bI_{k,\cs}(X)$ has a
filling $V \in \bI_{k+1,\cs}(X)$ with
\[
  \M(V) \le C\,\M(T)^{1 + \del}.
\]
\end{Thm}  

The assumption on the asymptotic rank will be used exclusively through 
Proposition~\ref{Pro:p-filling}. To illustrate the general principle, we will
write out the proof as if this result were to hold for arbitrary $\nu \ge 1$.

\begin{proof}
First, we consider a cycle $R \in \bI_{k,\cs}(X)$ with $\diam(R) \le \beta\,\M(R)^{1/k}$, where $\beta > 0$ is as above. 
By assumption, there exist $n,c,r > 0$ such that for $s > r$, every compact subset of $X$ has a finite $cs$-bounded covering with $s$-multiplicity at most $n+1$. 
If $\M(R) \le r^{1/\del}$, then
\[
\Fill(R) \le \gam_k\,\M(R)^{1+1/k} 
\le \gam_k\,r^{1/(k\del) - 1}\,\M(R)^{1 + \del}
\]
by Theorem~\ref{Thm:eucl-isop}. We now assume that $\M(R)^\del > r \ge 1$ and choose scales 
\[
\M(R)^{1/k} = s_0 > s_1 > \ldots > s_l = \M(R)^\del
\]
such that $s_{h-1}/s_h \le \M(R)^\del$ for $h = 1,\ldots,l$, where $l < 1/(k\del)$.
We show by induction that for $h = 1,\ldots,l$ there exist finite families $\cP_h,\cZ_h$ of $k$-cycles and decompositions
\[
R = \sum_{\cP_1} P + \ldots + \sum_{\cP_h} P + \sum_{\cZ_h} Z
\]
such that $\sum_{\cZ_h} \M(Z) \le B^h\,\M(R)$, every $P \in \cP_h$ is piecewise minimizing with $\mesh(P) \le B s_h$, and every $Z \in \cZ_h$ satisfies $\diam(Z) \le Bs_h$, where $B = B(k,n,c)$ is the constant from Theorem~\ref{Thm:approx}.
For $h=1$, this follows by applying the theorem with $s = s_1$ to $R$ (thus $\cP_1$ has a single member).
Let now $h \in \{2,\ldots,l\}$. 
Given the decomposition for $h-1$, assume that $\cZ_{h-1} = (Z_j)_j$, and apply Theorem~\ref{Thm:approx} with $s = s_h$ to write each $Z_j$ as a sum of a piecewise minimizing cycle $P_j$ with $\mesh(P_j) \le Bs_h$ and cycles $Z_{j,i}$ with $\diam(Z_{j,i}) \le Bs_h$.
Set $\cP_h := (P_j)_j$ and $\cZ_h := (Z_{j,i})_{j,i}$. 
This yields the decomposition for $h$, and 
\[
\sum_{\cZ_h}\M(Z) = \sum_{j,i} \M(Z_{j,i}) \leq \sum_j B\, \M(Z_j) \leq B^h\, \M(R),
\]
as desired. 

We now fill the individual cycles of the decomposition 
obtained after the last step. By construction, for $h \in \{1,\ldots,l\}$, 
every $P \in \cP_h$ stems from some $Z \in \cZ_{h-1}$, 
or from $Z = R$ in case $h = 1$, so that the vertices of $P$ lie in $\spt(Z)$ 
and $|P|_1 \le B s_h^{-k} \,\M(Z)$.
We assume for convenience that $B \ge \beta$.
By Proposition~\ref{Pro:p-filling},
\begin{align*}
\Fill(P) &\le C_{k+1} \diam(Z) \mesh(P)^{\nu - 1} \,|P|_1 \\
&\le C_{k+1} \cdot Bs_{h-1} \cdot (Bs_h)^{\nu-1} \cdot Bs_h^{-k}\, \M(Z),
\end{align*}
where $s_{h-1} s_h^{\,\nu-1-k} \le s_{h-1} s_h^{-1} \le \M(R)^\del$.
Since $\sum_{\cZ_{h-1}} \M(Z) \le B^{h-1}\,\M(R)$, or $\M(Z) = \M(R)$ if $h = 1$, we have
\[
\sum_{h=1}^l \sum_{\cP_h} \Fill(P) \le \sum_{h=1}^l C_{k+1} B^{\nu + h}\,\M(R)^{1+\del};
\]
recall that $l < 1/(k\del)$. 
Lastly, if $Z \in \cZ_l$, then $\diam(Z) \le B s_l = B\,\M(R)^\del$, thus Proposition~\ref{Pro:coning} gives 
\[
\sum_{\cZ_l} \Fill(Z) \le \frac{B\,\M(R)^{\del}}{k+1} \sum_{\cZ_l} \M(Z)
\le \frac{B^{l+1}}{k+1}\,\M(R)^{1 + \del}.
\]
This proves the desired $\del$-isoperimetric inequality with a constant $C'$ for all cycles $R \in \bI_{k,\cs}(X)$ 
with $\diam(R) \le \beta\,\M(R)^{1/k}$.

Let now $T \in \bI_{k,\cs}(X)$ be a general cycle and consider the decomposition of $T$ from Proposition~\ref{Pro:decomp}. 
Put $\mu := \eps/(1+\lam)$, so that~(2) and~(3) give
\[
\M(T') + \mu \sum_{i=1}^N \M(R_i) \le \M(T).
\]  
As we have shown, there exist
$V_1,\ldots,V_N \in \bI_{k+1,\cs}(X)$ such that $\d V_i = R_i$ and
\[
\M(V_i) \le C'\,\M(R_i)^{1+\del}.
\]
Then $V' := \sum_{i=1}^N V_i$ is a filling of $T - T'$, and
\begin{align*}
C'\,\M(T')^{1+\del} + \mu^{1+\del}\,\M(V')
&\le C' \biggl( \M(T')^{1+\del} + \sum_{i=1}^N \bigl( \mu\,\M(R_i) \bigr)^{1+\del} \biggr)\\ 
&\le C' \biggl( \M(T') + \mu \sum_{i=1}^N \M(R_i) \biggr)^{1+\del} \\
&\le C'\,\M(T)^{1+\del}.
\end{align*} 
Upon iterating this inequality a finite number of times, we can assume
$\M(T')$ to be small enough such that Theorem~\ref{Thm:eucl-isop} provides
a filling $S \in \bI_{k+1,\cs}(X)$ of $T'$ satisfying
\[
\mu^{1+\del}\,\M(S) 
\le \mu^{1+\del}\gam_k\,\M(T')^{1 + 1/k}
\le C'\,\M(T')^{1+\del}.
\]
Then $V := S + V'$ is a filling of $T$ with $\M(V) \le
C\,\M(T)^{1+\del}$ for $C := C'/\mu^{1+\del}$. 
\end{proof}

\subsection*{Acknowledgements}
The first two authors thank the Hausdorff Research Institute for Mathematics (HIM) in Bonn for hospitality during the trimester program {\em Metric Analysis} from January to April 2025.
The HIM is funded by Deutsche Forschungsgemeinschaft (DFG), project EXC 2047, no.~390685813.
We thank Cornelia Dru\c{t}u and Panos Papasoglu for useful discussions.



\end{document}